\documentclass[11pt,reqno]{amsart} 
\usepackage[utf8]{inputenc}
\usepackage{amsfonts,amsmath,amsthm,amssymb}
\usepackage[font=small]{subcaption}
\usepackage[justification=centering]{caption}
\usepackage{color}
\usepackage{enumitem}   
\usepackage{hyperref}
\usepackage{mathtools}
\usepackage{xcolor}
\usepackage{thm-restate}
\usepackage{thmtools}
\usepackage[utf8]{inputenc}
\usepackage{geometry,amsmath,amssymb,graphicx,float,tikz,mathabx,mathrsfs, amsthm, ytableau, extarrows, hyperref}
\usepackage{verbatim}
\usepackage{graphicx}
\usepackage{mathpazo}

\newtheorem{thmx}{Theorem}

\newcommand{\PF}{\operatorname{PF}}
\newcommand{\PPF}{\operatorname{PPF}}
\newcommand{\UIPF}{\operatorname{UPF}}

\newcommand{\ppof}[1]{#1'}
\newcommand{\sii}{\genfrac{\{}{\}}{0pt}{}} 
\newcommand{\uIPFn}{\operatorname{UPF}_n}
\newcommand{\uIPF}[2]{\operatorname{UPF}_{#1}(#2)} 

\newcommand{\T}{\mathcal{T}}

\newenvironment{acknowledgements} {\begin{abstract}} {\end{abstract}}

\newtheorem{theorem}{Theorem}[section]
\newtheorem{proposition}[theorem]{Proposition}
\newtheorem{lemma}[theorem]{Lemma}
\newtheorem{corollary}[theorem]{Corollary}

\theoremstyle{definition}
\newtheorem{definition}[theorem]{Definition}

\newtheorem{remark}[theorem]{Remark}

\newtheorem{example}[theorem]{Example}

\newcounter{cases}
\newcounter{subcases}[cases]

\newcommand\restr[2]{{
  \left.\kern-\nulldelimiterspace 
  #1 
  \vphantom{\big|} 
  \right|_{#2} 
  }}
  
\newcommand\dyckpath[5]{
  \begin{scope}[local bounding box=#4]
    \fill[white]  (#1) rectangle +(#2,#2);
    \fill[red!5!white] (#1) foreach \dir in {#3}{-- ++(\dir*90:1)} |- (#1);
    \path[fill] (#1) foreach \i [count=\j] in {1,...,#5}{ +(\i - .5,0) node[anchor=north]{\j} \ifnum\i>#2 circle (1pt) \fi};
    \draw[help lines] (#1) grid +(#2,#2);
    \draw[line width=2pt] (#1) foreach \dir in {#3}{ -- ++(\dir*90:1)};
  \end{scope}
}

\title{Unit-Interval Parking Functions and the Permutohedron}

\author[Chaves Meyles]{Lucas Chaves Meyles}
\address[L.~Chaves Meyles]{Department of Mathematics, University of California Los Angleles, Los Angeles, CA 90095}
\email{\textcolor{blue}{\href{mailto:lchavesmeyles@gmail.com}{lchavesmeyles@gmail.com}}}

\author[Harris]{Pamela E. Harris}
\address[P.~E. Harris]{Department of Mathematical Sciences, University of Wisconsin-Milwaukee, Milwaukee, WI 53211}
\email{\textcolor{blue}{\href{mailto:peharris@uwm.edu}{peharris@uwm.edu}}}

\author[Jordaan]{Richter Jordaan}
\address[R.~Jordaan]{Department of Mathematics, Massachusetts Institute of Technology, Cambridge, MA 02139}
\email{\textcolor{blue}{\href{mailto:rjordaan@mit.edu}{rjordaan@mit.edu}}}

\author[Rojas Kirby]{Gordon Rojas Kirby}
\address[G.~Rojas Kirby]{Department of Mathematics, San Diego State University, San Diego, CA 92182}
\email{\textcolor{blue}{\href{mailto:}{gkirby@sdsu.edu}}}

\author[Sehayek]{Sam Sehayek}
\address[S.~Sehayek]{Department of Mathematics, University of California Santa Barbara, Santa Barbara, CA 93117}
\email{\textcolor{blue}{\href{mailto:}{ssehayek@ucsb.edu}}}

\author[Spingarn]{Ethan Spingarn}
\address[E.~Spingarn]{Department of Mathematics, Amherst College, Amherst, MA 01002}
\email{\textcolor{blue}{\href{mailto:espingarn23@amherst.edu}{espingarn23@amherst.edu}}}

\begin{document}

\begin{abstract}
Unit-interval parking functions are subset of parking functions in which cars park at most one spot away from their preferred parking spot. In this paper, we characterize unit-interval parking functions by understanding how they decompose into prime parking functions and count unit-interval parking functions when exactly $k<n$ cars do not park in their preference.
This count yields an alternate proof of a result of  Hadaway and Harris establishing that unit-interval parking functions are enumerated by the Fubini numbers.
Then, our main result, establishes that for all integers $0\leq k<n$, the unit-interval parking functions of length $n$ with displacement $k$ are in bijection with the $k$-dimensional faces of the permutohedron of order $n$.
We conclude with some consequences of this result.

\end{abstract}

\maketitle

\section{Introduction}

Throughout we let $\mathbb{N}\coloneqq\{1,2,3,\ldots\}$ and whenever $n\in\mathbb{N}$ we let $[n]\coloneqq\{1,2,\ldots,n\}$. We recall that 
a \textbf{parking function} of length $n$ is a $n$-tuple $\alpha=(a_1,a_2,\ldots,a_n)\in[n]^n$ whose nondecreasing rearrangement $\beta=(b_1,b_2,\ldots,b_n)$ satisfies $b_i\leq i$ for all $i\in[n]$.
One can also define these combinatorial objects through the following parking scenario. 
A parking function of length $n$ is an assignments of preferences for $n$ cars attempting to park on a one-way street with $n$ parking spots, such that every car can find a parking spot using the following parking protocol:

Cars park one at a time (in order 1 through $n$), if a car's preferred parking spot is occupied upon its arrival, then the car proceeds down the one-way street and parks in the first unoccupied parking spot beyond its preference. 
For example, $(2,1,3,3)$ is a parking function in which car one prefers spot two and parks there, car two prefers spot one and parks there, while both cars three and four prefer spot three, and car three parks in spot three, but car four must proceed and park in spot four. However, $(1,3,3)$ is not a parking function as car three is unable to park in one of the first three parking spots.
We let $\PF_n$ denote the set of parking functions of length $n$.

Parking functions were introduced by Konheim and Weiss and by Pyke and Riordan in their study of linear probing resolution strategies for random hashing functions \cite{KonheimAndWeiss, Riordan}. 
They established that $|\PF_n|=(n+1)^{n-1}$. 
Since their introductions, many have studied parking functions and their connections to graph theory, probability, hyperplane arrangements, volume of polytopes, and more \cite{stanley1997parking,stanley2002polytope}. For a wonderful survey on parking functions we recommend \cite{yan2015parking}.
Others have generalized the concept of parking functions in a number of directions to include the street having more parking spots than cars, a street with some spots occupied prior to the cars arriving, parking protocols allowing cars to back up when finding their preferred spot occupied, cars having a set of preferences rather than a single preference, and cars with varying lengths \cite{parkingcompletion,knaples,colaric2021interval,knaplesperms,CountingPAandPS}.
For the reader interested in open problems related to parking functions we recommend \cite{Choose}.

Parking functions have also been studied based on their statistics: enumerations that describe certain properties. For example, Gessel and Seo enumerated parking functions based on the number of lucky cars, those which park in their preferred spot \cite{GesselSeo}, and Schumacher enumerated parking functions based on the number of ascents, descents, and ties,  \cite{Schumacher}, and Adeniran and Pudwell gave enumerations of parking functions avoiding certain patterns \cite{adeniran2022pattern}.

A statistic central to our study is called the \emph{displacement statistic}, defined as follows.
If $\alpha=(a_1,a_2,\ldots,a_n)\in\PF_n$, then, for any $i \in [n]$, car $i$ has preference $a_i$ and parks in spot $s_i$ and we define the \textbf{displacement of car $i$} by $d_i=s_i - a_i$, which measures the distance between where car $i$ actually parked and where it wanted to park.
Then the \textbf{displacement}\footnote{Note that by thinking of parking functions as encoding the collision resolution scheme of a random hashing function,  the displacement of an item is the number of linear probes required to insert it into the hash table. Thus the total displacement encodes the total number of linear probes during the insertion of all items into the hash table. 
} of a parking function $\alpha$ is 
\[ D(\alpha) = \sum_{i \in [n]} d_i.\] 
For example, $D(2,1,3,3)=1$ as the only car which is displaced is the fourth car, and it is displaced by one. Note that $0\leq D(\alpha)\leq \binom{n-1}{2}$ for all $\alpha\in\PF_n$, where the left equality is achieved when $\alpha$ is a permutation of $[n]$  and the right equality is achieved when $\alpha$ is the all ones preference list.
Aguillon et.~al \cite{tower} showed that the set of parking functions of length $n$ for which $D(\alpha)=1$ is in bijection with the set of ideal states in the famous Tower of Hanoi game with $n+1$ disks and $n+1$ pegs, both sets being enumerated by the Lah numbers \cite[\href{http://oeis.org/A001286}{A001286})]{OEIS}.

In our work we consider parking functions with displacement $k$.  However, the value~$k$ can arise in many ways, i.e.~parking functions with displacement $k$ could have a single car being displaced by $k$ or $k$ distinct cars each being displaced by one. 
Thus, to study parking functions with displacement $k$, one must consider all integer partitions of $k$, as well as which cars contribute a part to the partition.
To record both pieces of information we introduce the  \textbf{displacement vector} of $\alpha\in\PF_n$, defined by $V(\alpha)=(d_1,d_2,\dots,d_n)$. 
Arranging the entries of the displacement vector in nonincreasing order allows us to think of this vector as an integer partition, which we henceforth refer to as the \textbf{displacement partition}.
More precisely, given $\alpha\in\PF_n$ the displacement partition $\lambda\vdash D(\alpha)$ is the integer partition of $D(\alpha)$ consisting of the nonzero entries of $V(\alpha)$ rearranged into nonincreasing order. We remark that  as we consider parking functions with fixed displacement, deleting zeros and rearranging the entries of the displacement vector into nondecreasing order, is inconsequential in our analysis, as the partition simply allows us to determine the number of cars that are displaced and by how much.  

For example, if $\alpha=(2,2,1,1)$, then $D(\alpha)=4$, $V(\alpha)=(0,1,0,3)$, and $\lambda=(3,1)$.
In the case that $\alpha$ is the all ones preference list, then $D(\alpha)=\binom{n-1}{2}$, $V(\alpha)=(0,1,2,\ldots,n-1)$, and $\lambda=(n-1,n-2,\dots, 1)$.

In this paper, we study the set of \textbf{unit-interval parking functions}\footnote{We remark that a unit-interval parking function can be thought of as an optimal outcome of the linear probing collision resolution scheme: any collision in the hash table is resolved in a single linear probe.}
 of length $n$, denoted by $\UIPF_n$, defined by Hadaway and Harris as the set of parking functions where each car's displacement is at most one \cite{Hadaway}. 
They showed that unit-interval parking functions are enumerated by the Fubini numbers\footnote{Also known as the ordered Bell numbers \cite[\href{http://oeis.org/A000670}{A000670})]{OEIS}.} \cite[Theorem~5.12]{Hadaway}:
\begin{align}
    |\UIPF_{n}|=\sum _{{k=0}}^{n}\sum _{{j=0}}^{k}(-1)^{{k-j}}{\binom  {k}{j}}j^{n}.\label{eq;Fubini numbers}
\end{align}
Note that unit-interval parking functions 
are a specific type of the broader class of interval parking functions introduced and studied in  \cite{colaric2021interval}.
We also remark that there is a correspondence between weakly increasing unit-interval parking functions and Dyck paths of height at most one \cite{baril2018dyck}, which we detail and leverage in Remark~\ref{rem:characterize}.
Note that by definition, if $\alpha$ is a unit-interval parking function, then the entries of $V(\alpha)$ are zero or one and $\lambda = (1,1,\ldots,1)\vdash D(\alpha)$. 
We let $\uIPF{n}{k}=\{\alpha\in\uIPFn\,:\, D(\alpha)=k\}$ denote the set of unit-interval parking functions of length $n$ with displacement $k$. 

The main result of this paper is the  surprising connection between $\uIPF{n}{k}$ and the permutohedron. Before formally stating this result, we recall that the permutohedron (of order $n$), denote by $P(n)$, 
is defined by taking the convex hull of the permutations of the elements in $[n]$. 
Note that $P(n)$ is an $(n-1)$-dimensional polytope embedded in $n$-dimensional space. We  now state our main result.

\begin{thmx}[\ref{thm: pf to pphedron}]
For all integers $0\leq k<n$, unit-interval parking functions of length $n$ with displacement~$k$ are in bijection with the $k$-dimensional faces of the permutohedron of order $n$.
\end{thmx}

\begin{figure}[H]
\centering
\resizebox{2.25in}{!}{
\begin{tikzpicture}[z={(0,0,.5)}, line join=round, scale = 2]

\node at (2, 0) (123) {$123$};
\node at (1, 1.732) (132) {$132$};
\node at (-1, 1.732) (231) {$231$};
\node at (-2, 0) (321) {$321$};
\node at (-1, -1.732) (312) {$312$};
\node at (1, -1.732) (213) {$213$};


\draw[line width = 0.5mm] (123) -- (132) node[color = red, fill = white, midway] (122) {$122$} -- (231) node[color = red, fill = white, midway] (131) {$131$} -- (321) node[color = red, fill = white, midway] (221) {$221$} -- (312) node[color = red, fill = white, midway] (311) {$311$} -- (213) node[color = red, fill = white, midway] (212) {$212$} -- (123) node[color = red, fill = white, midway] (113) {$113$};

\node[color = teal] at (0, 0) (112) {$112$};
\draw[color = cyan] (311) -- (112);
\draw[color = cyan] (212) -- (112);
\draw[color = cyan] (113) -- (112);
\draw[color = cyan] (122) -- (112);
\draw[color = cyan] (131) -- (112);
\draw[color = cyan] (221) -- (112);

\end{tikzpicture}
}
\caption{The permutahedron $P(3)$ with elements of $\uIPF{3}{0}$, $\uIPF{3}{1}$, and $\uIPF{3}{2}$ labeling vertices, edges, and faces, respectively.}
\label{fig: 3-permutohedron}
\end{figure}
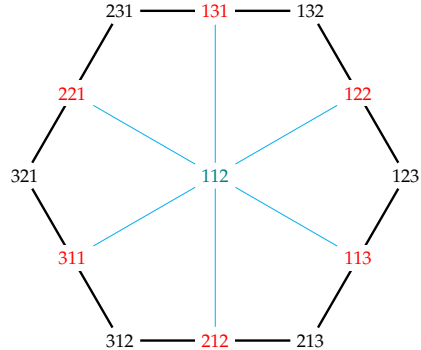

Figure~\ref{fig: 3-permutohedron} illustrates\footnote{To ease notation within our figures and elsewhere, we take the liberty of omitting commas and parenthesis from the parking function notation in certain instances.}  Theorem~\ref{thm: pf to pphedron} when $n=3$. Note that the vertices of $P(3)$ are the permutations, which are precisely the elements of $\uIPF{3}{0}$, the edges of $P(3)$ are labeled by the elements of $\uIPF{3}{1}=\{(1,1,3),(1,2,2),(1,3,1),(2,1,2),(2,2,1),(3,1,1)\}$, and the hexagonal face of $P(3)$ corresponds to the element of 
$\uIPF{3}{2}=\{(1,1,2)\}$, which is the unique unit-interval parking function of length $3$ that displaces two cars by one.

In Section~\ref{sec:pporder&primes}, we introduce the needed technical background to make our approach precise and establish Theorem~\ref{thm: pf to pphedron}. 
This includes introducing a particular rearrangement of parking functions that preserves displacement partitions. 
In the case of unit-interval parking functions, this is a canonical choice of rearrangement into weakly increasing order. Using this rearrangement, we explore how displacement of a parking function is completely realized by what we call component primes, which are prime parking functions, defined by Gessel, as parking functions of length $n$ whose removal of any single instance of one returns a parking function of length $n-1$ (c.f. Exercise 5.49 in \cite{stanley1999enumerative}). 
This prime decomposition of parking functions---expanded upon in \cite{PFFixedDisplacement}---is a powerful reframing 
that is key to proving Theorem~\ref{thm: pf to pphedron} in Section~\ref{sec:permutohedron}. 
We conclude by showing that 
the structure of the permutohedron and the action of the symmetric group on it illuminates a natural structure within unit-interval parking functions that is respected by this symmetric group action.

\begin{remark}\label{rem: fubini rankings}
We conclude this introduction by noting that Hadaway and Harris gave a bijection between unit-interval parking functions of length $n$ and Fubini rankings, which are $n$-tuples giving the rankings in a competition among $n$ competitors where ties are allowed.
Fubini rankings were known to be enumerated by the Fubini numbers given in \cite{Hadaway}.
Unlike unit-interval parking functions, Fubini rankings are invariant under the action of the symmetric group. Namely, any permutation of a Fubini ranking yields a Fubini ranking. 

Thinking of Fubini rankings as weak orderings on an $n$ element set, it was remarked in \cite{ovchinnikov2004weak} that Fubini rankings 
are in correspondence with the faces of all dimensions of a permutohedron, though no explicit bijection was provided.
Given that unit-interval parking functions are not permutation invariant, it is more surprising that unit-interval parking functions continue to be in bijection with the $k$-dimensional faces of the permutohedron in such a natural way.

\end{remark}

\section{Parking Order and Prime Decomposition}\label{sec:pporder&primes}
We begin by defining parking-ordered parking functions and the prime decomposition of a parking function. These concepts lay the groundwork for our bijections in later sections. We remark that these two ideas have rich properties and applications; we point the interested reader to \cite{PFFixedDisplacement}. However, given our objective of proving Theorem \ref{thm: pf to pphedron}, we focus our results to the set of unit-interval parking functions.

\begin{definition}\label{def: partition-preserving order}
We say that a parking function $\alpha = (a_1, \dots, a_n) \in PF_n$ is \textbf{parking-ordered} if $a_i \leq i$ for all $i \in [n]$. 

\end{definition}

Note that in a parking-ordered parking function, the car with preference $a_i$ parks in spot $i$. 
Moreover, any parking function $\alpha = (a_1, \dots, a_n)$ can be rearranged into a parking-ordered parking function $\alpha'$ (also referred to as the \textbf{parking rearrangement} of $\alpha$) as follows. Let $Sym_n$ denote the permutations of the set $[n]$, written in one -line notation within an $n$-tuple. Let $\sigma^{-1} = (s_1, \dots, s_n)\in Sym_n$ 
record where each car parks under the preference list $\alpha$. I.e. the $i$th car of $\alpha$ parks in spot $s_i$. In the parking function literature such as \cite{colaric2021interval}, this permutation is often referred to as the \textbf{parking outcome} of $\alpha$. Then, $\alpha'=\sigma(\alpha)=(a_{\sigma(1)},\dots, a_{\sigma(n)})$, where $\sigma\in Sym_n$ acts on an $n$-tuple by permuting its coordinates.
In other words, if the car with preference $a_i$ parks in spot $s_i$, then the $s_i$-th entry of of $\ppof{\alpha}$ is equal to $a_i$. 

For example, if $\alpha=(2,5,4,6,1,1)$, then the parking outcome of $\alpha$ is $\sigma^{-1}=(2,5,4,6,1,3)$, with $\sigma=(5,1,6,3,2,4)$.  The parking rearrangement of $\alpha$ is $\ppof{\alpha}=\sigma(\alpha)=(1,2,1,4,5,6)$.

Moreover, the parking rearrangement of a parking function is parking-ordered and respects the displacement partition.
\begin{lemma}\label{lem: pp_dp_preserving}

Let $\alpha = (a_1, \dots, a_n) \in PF_n$ and $\alpha'$ be the parking rearrangement of $\alpha$. Then, $\alpha'$ is a parking function that has the same displacement partition as $\alpha$. Furthermore, $\alpha'$ is in parking order.
\end{lemma}

\begin{proof}
By Definition \ref{def: partition-preserving order}, $\ppof{\alpha}$ is a permutation of the elements of $\alpha$ so $\alpha'$ is a parking function of length $n$. Next, consider the displacement vector of $\ppof{\alpha}$,
\[
V(\alpha')=(1-a_{\sigma(1)},\dots, n-a_{\sigma(n)})=(s_{\sigma(1)}-a_{\sigma(1)},\dots,s_{\sigma(n)}-a_{\sigma(n)})=\sigma(V(\alpha)).
\]
Hence, the displacement vector $V(\alpha')$ is just a permutation of $V(\alpha)$ so that $\alpha$ and $\alpha'$ share the same displacement partition.

We now show that $\ppof{\alpha}$ is in parking order. Notice that it is impossible for a car with preference $a_i$ to park in a spot which is less than $a_i$. Hence, $a_i \leq s_i$ for all $i \in [n]$. Applying $\sigma$ to $\alpha$ and the parking outcome, we obtain the inequalities $a_{\sigma(i)}\leq s_{\sigma(i)}=i$ for each $i$, proving that $\ppof{\alpha}$ is in parking order. 

\end{proof}

A major advantage of the parking rearrangement of a parking function $\alpha$ is that one can easily identify what we call the component primes of $\alpha$ from this reordering. To clarify this idea, we begin by recalling prime parking functions and breakpoints of a parking function from \cite{GILBEY1999351}. We also introduce a useful set operation called the pipe. 
\begin{definition}\label{def: breakpoint}
We say a parking function $\alpha=(a_1,\dots, a_n)$ has a \textbf{breakpoint} at $k$ if exactly $k$ cars want to park in the first $k$ spaces, i.e. $|\{i\mid a_i\leq k\}|=k$. We say that a parking function of length $n$ is \textbf{prime} if it has a single breakpoint at $n$.
\end{definition}

For example $(1,1,2)$ is a prime parking function with a single break point at 3, while $(1,1,3)$ is not prime and has break points at 2 and 3. Also, observe that any rearrangement of a prime parking function is prime.

We let $\PPF_n$ denote the set of all prime parking functions of length $n$. It is well-known that $|\PPF_n|=(n-1)^{n-1}$, see \cite{armstrong2016rational,GILBEY1999351,stanley1999enumerative} for proofs. In the case of unit-interval parking functions there is a unique prime unit-interval  parking function of length $n$.

\begin{lemma}\label{lem: unique_prime_unit_interval_pf}
Let $\alpha\in\UIPF_{n}$. Then $\alpha$ is prime if and only if $\alpha=(1, 1, 2, 3, \dots, n-1)$.

\end{lemma}
\begin{proof}

If $\alpha=(1,1,2,3,\ldots,n-1)$, then each car except the first is displaced exactly one space, and $\alpha$ has a single breakpoint at $n$ so that it is prime. 

Conversely, suppose that $\alpha=(a_1,\dots, a_n)\in \UIPF_n$ is prime. Let $\alpha'=(a_1',\dots, a_n')$ be its parking order rearrangement. Then $\alpha'$ is prime and Lemma \ref{lem: pp_dp_preserving} implies $\alpha'\in \UIPF_n$. Since $\alpha'$ is a unit-interval parking function, $i-1\leq a_i'\leq i$ for each $i$. Since $\alpha'$ is also prime the only breakpoint is at $n$, i.e. $|\{k\mid a_i\leq k \}|>k$ for each $1\leq k<n$. Since $\alpha'$ is in parking order $a_1'=1$, which implies $a_2'=1$ and this implies $a_3'=2$. Continuing in the fashion we have $\alpha'=(1,1,2,3,\ldots,n-1)$.

If $\alpha=\ppof{\alpha}$ we are done. So suppose $\alpha\neq \ppof{\alpha}$. Hence there is some minimal index $i$ such that $a_i\neq \ppof{a}_i$. 
Thus, $a_i>\max(1, i-1)$ and since only the first $i-1$ spots have already been occupied, car $i$ parks in its preference. 
This contradicts the fact that $\alpha$ has the same displacement partition as $\alpha'$, in which each car except the first is displaced exactly one spot.

Therefore, $\alpha=\ppof{\alpha}=(1,1,2,\dots, n-1)$ as desired. 
\end{proof}

Lemma~\ref{lem: unique_prime_unit_interval_pf} shows that any prime unit-interval parking function displaces $n-1$ cars each by one unit, immediately implying the following.

\begin{corollary}
Let $\alpha\in\UIPF_{n}$. Then $\alpha$ is a prime parking function 
if and only if $D(\alpha)= n-1$.
\end{corollary}

 \begin{definition}
    For finite tuples of integers $A=(a_1,\dots,a_m)$ and $B=(b_1,\dots,b_n)$ we write $B+ k$ to mean $(b_1+k,\dots,b_n+k)$ and we denote the \textbf{pipe} $A|B$ as the concatenation of $A$ with $B+|A|$. This definition inductively extends to the pipe of a finitely many tuples. 
\end{definition}
For example, if $A=(1,1,2)$, $B=(1,1,2,3,4)$, and $C=(1,1)$ then the pipe $A|B|C$ is \[A|B|C=(1,1,2,4,4,5,6,7,9,9).\]

Now, we outline the procedure to decompose any parking-ordered unit-interval parking function into prime parking functions. This process can be done with slight modification to the following outline for any parking-ordered parking function---this general setting is discussed in \cite{PFFixedDisplacement} and leveraged to enumerate parking functions with arbitrary displacement partitions. 

\begin{remark}[Prime Decomposition]\label{prime decomp for Unit}
    
Let $\alpha = (a_1, \dots, a_n) $ be a parking-ordered  unit-interval parking function of length $n$. First we partition $\alpha$ into blocks corresponding to prime parking functions, one for each break point. We call these blocks the \textbf{component primes} of $\alpha$. 
Since our underlying parking function is unit-interval, each block of size $r$ is a subsequence of the form $U+(i-1)$ where $i-1$ is either zero or a breakpoint of $\alpha$ and $U=(1,1,2,\dots, r-1)$, which is the unique  prime unit-interval parking function of length~$r$ given by Lemma~\ref{lem: unique_prime_unit_interval_pf}.

Denoting the underlying prime parking functions as $P_1,P_2,\dots, P_k$ (where $k$ is the number of breakpoints) we obtain an ordered list $P_1,P_2,\dots,P_k$ of prime parking functions such that $\alpha=P_1|P_2|\cdots|P_k$ (the pipe of $P_1,P_2,\dots,P_k$). 
 \end{remark}

For an arbitrary parking function, the component primes that arise in its parking rearrangement will also be referred to as component primes. We say a component prime parking function is nontrivial if it has length at least $2$. Otherwise, its displacement is $0$ and as such does not contribute to the displacement partition. All nontrivial  prime parking functions exhibit displacement.

\begin{theorem}\label{thm: upf decomposition}
Let $\alpha$ be a unit-interval parking function in parking order. Then $\alpha$ can be decomposed uniquely into prime unit-interval parking functions. In particular, we can characterize unit-interval parking functions as those parking functions whose component primes are in the form of Lemma \ref{lem: unique_prime_unit_interval_pf}.
\end{theorem}

\begin{proof}
The first part of the theorem statement is the content of Remark \ref{prime decomp for Unit} and as noted this can be extended to an arbitrary unit-interval parking function by passing through to the parking rearrangement. The decomposition follows directly from the fact that $(1,2,\dots,n) - V(\alpha)=\alpha$ for parking-ordered parking functions. 
Even for the general case treated in \cite{PFFixedDisplacement}, any parking-ordered parking function has a unique decomposition into prime parking function components.
Since the decomposition respects the displacement, each component in the decomposition of $\alpha$ is itself a unit-interval parking function. By Lemma \ref{lem: unique_prime_unit_interval_pf}, a prime unit-interval parking function is determined by its length and has a very specific form. 
So, a parking function can be determined to be unit-interval by investigating the prime components in its parking rearrangement. 
\end{proof}

\begin{remark}\label{rem:characterize}
    Theorem \ref{thm: upf decomposition} gives a full characterization of unit-interval parking functions. One quick way to check that a parking function is unit-interval is to look at its associated labeled Dyck path. For a description of the bijection we refer the interested reader to \cite{max}. Following the insight of Theorem \ref{thm: upf decomposition} in this context, if the Dyck path is height 1 and the labels in each prime component (areas between returns to the diagonal---these returns are exactly breakpoints as defined above) are increasing, it is unit-interval. The converse is also true as illustrated in Example \ref{ex:DP1}.
\end{remark}


\begin{figure}
\centering

\begin{subfigure}{0.3\textwidth}
    \begin{tikzpicture}[scale = 0.5]
  \dyckpath{0, 0}{7}{1, 0, 1, 1, 0, 1, 0, 1, 0, 0, 1, 1, 0, 0}{dyck}{7};
  \draw[color = blue, dashed] (0, 0) -- (7, 7);

  \node[label=left:$6$] (b1) at (1.25, 0.5) {};
  \node[label=left:$1$] (b2) at (2.25, 1.5) {};
  \node[label=left:$2$] (b3) at (2.25, 2.5) {};
  \node[label=left:$3$] (b4) at (3.25, 3.5) {};
  \node[label=left:$7$] (b5) at (4.25, 4.5) {};
  \node[label=left:$4$] (b6) at (6.25, 5.5) {};
  \node[label=left:$5$] (b7) at (6.25, 6.5) {};
  \end{tikzpicture}

    \caption{}
    \label{fig:first}
\end{subfigure}
\hfill
\begin{subfigure}{0.3\textwidth}
    \begin{tikzpicture}[scale = 0.5]
  \dyckpath{0, 0}{7}{1, 1, 0, 1, 1, 1, 0, 0, 0, 1, 1, 0, 0, 0}{dyck}{7};
  \draw[color = blue, dashed] (0, 0) -- (7, 7);

  \node[label=left:$4$] (b1) at (1.25, 0.5) {};
  \node[label=left:$7$] (b2) at (1.25, 1.5) {};
  \node[label=left:$2$] (b3) at (2.25, 2.5) {};
  \node[label=left:$3$] (b4) at (2.25, 3.5) {};
  \node[label=left:$5$] (b5) at (2.25, 4.5) {};
  \node[label=left:$1$] (b6) at (5.25, 5.5) {};
  \node[label=left:$6$] (b7) at (5.25, 6.5) {};
  \end{tikzpicture}
    \caption{}
    \label{fig:second}
\end{subfigure}
\hfill
\begin{subfigure}{0.3\textwidth}
    
  \begin{tikzpicture}[scale = 0.5]
  \dyckpath{0, 0}{7}{1, 1, 0, 1, 0, 1, 0, 0, 1, 1, 0, 1, 0, 0}{dyck}{7};
  \draw[color = blue, dashed] (0, 0) -- (7, 7);
  \node[label=left:$1$] (b1) at (1.25, 0.5) {};
  \node[label=left:$4$] (b2) at (1.25, 1.5) {};
  \node[label=left:$2$] (b3) at (2.25, 2.5) {};
  \node[label=left:$3$] (b4) at (3.25, 3.5) {};
  \node[label=left:$5$] (b5) at (5.25, 4.5) {};
  \node[label=left:$6$] (b6) at (5.25, 5.5) {};
  \node[label=left:$7$] (b7) at (6.25, 6.5) {};
  \end{tikzpicture}  
    \caption{}
    \label{fig:third}
\end{subfigure}
        
\caption{Visualizing unit-interval parking function example and non-examples.}

   \label{fig: DP1}
\end{figure}
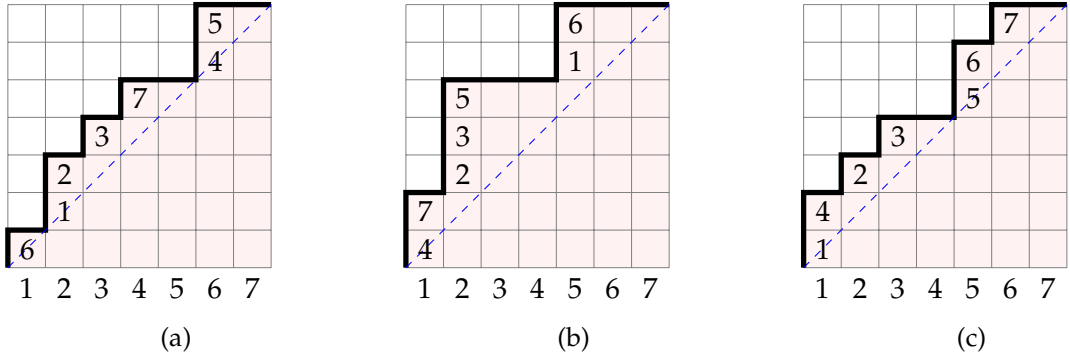
\begin{example}\label{ex:DP1}

Figure \ref{fig:first} illustrates the labeled  Dyck path associated with $(2, 2, 3, 6, 6, 1, 4)$, a non-prime unit-interval parking function. Its increasing rearrangement has prime parking decomposition $(1)|(1,1,2,3)|(1,1)=(1,2,2,3,4,6,6)$ corresponds to decomposition of the Dyck path into the 3 subpaths between consecutive returns to the diagonal.

Figure \ref{fig:second} depicts the labeled Dyck path associated with $(5, 2, 2, 1, 2, 5, 1)$. It is prime but not a unit-interval parking function, which is evident since it is not height one.

However, a labeled Dyck path of height less than one is not sufficient to guarantee that the corresponding parking functions is unit-interval. The parking function $(1, 2, 3, 1, 5, 5, 6)$ is not unit-interval , but corresponds to the Dyck path in  Figure \ref{fig:third} that is height $1$. Indeed, the labels in the first component prime are not strictly increasing.
\end{example}

\begin{remark}\label{rem: pp_order_unit_interval} We conclude by noting that since each unit-interval prime parking function is uniquely determined by its length and is always in nondecreasing order, Theorem~\ref{thm: upf decomposition} implies that the parking rearrangement of a unit-interval parking function agrees with the rearrangement into nondecreasing order. 
\end{remark}


\section{Permutohedron}\label{sec:permutohedron}

In this section, we work with the well-studied description of the permutohedron in terms of ordered set partitions (c.f \cite{billera1994iterated},\cite{oldbook}). Specifically, the $(n-k)$-faces are in bijection with ordered set partitions of $[n]$ into $1\leq k\leq n$ parts, which are {known to be enumerated by 
\cite[\href{https://oeis.org/A019538}{A019538}]{OEIS}: $T(n,k)=k!\sii{n}{k}$, where $\sii{n}{k}$ is the Stirling numbers of the second kind.}

Consequently, we show that there exists an equivariant bijection between the faces of the permutohedron and unit-interval parking functions proving Theorem \ref{thm: pf to pphedron} in a manner that reflects the prime decomposition of these parking functions. Namely, the number of component primes in the prime decomposition (i.e. number of breakpoints as in Definition~\ref{def: breakpoint} and first defined in \cite{GILBEY1999351}) corresponds to the number of blocks in the associated ordered partition, and the sizes of the primes corresponds to the sizes of the blocks.

\begin{definition}\label{def: ordered_set_partition}
An ordered list of pairwise disjoint nonempty subsets $(B_1, \dots, B_k)$, called blocks, with $B_i\subseteq[n]$  is an \textbf{ordered set partition} of $[n]$ if the union $\bigcup B_i=[n]$. Throughout we write ordered set partitions of $[n]$ with $k$ blocks as $B_1 /B_2/ \cdots / B_k$. We write $\T_n$ to denote the set of ordered set partitions.
\end{definition}
 It is well-known that $\T_n$ is enumerated by the Fubini numbers:
\begin{align}
|\T_n|&=\sum_{k=0}^n k!\sii{n}{k},\label{eq:Fubini numbers alt}
\end{align}
where $\sii{n}{k}$ denotes the Stirling numbers of the second kind. Note that equation \eqref{eq:Fubini numbers alt} is an alternate form for the Fubini numbers that is equivalent to the formula given in equation~\eqref{eq;Fubini numbers}.

 \begin{definition}\label{def:T_n to UPFn}
For an ordered set partition $B_1/B_2/\dots/B_k$, where the entries in each $B_i$ are in increasing order there exists a unique unit-interval prime parking function $P_i$ of length $|B_i|$, i.e. $P_i=(1,1,2,\dots,|B_i|-1)$. Let $\sigma\in Sym_n$ be the permutation $(B_1,\dots,B_k)$ in one-line notation. Define $\psi(B_1/B_2/\dots/B_k)=\sigma(P_1|P_2|\cdots|P_k)$. 
 \end{definition}
 
 \begin{definition}\label{def: UPFn to Tn}
     Define $\phi: \uIPFn\to \T_n$ as follows. For $\alpha\in\uIPFn $ denote $\sigma^{-1}=(s_1,\dots, s_n) $ as the parking outcome of $\alpha$. Let $\alpha$ have break points $b_1,\dots,b_k$. Define  \[\phi(\alpha)=\sigma(1),\dots,\sigma(b_1)/\sigma(b_1+1),\dots,\sigma(b_2)/\dots/\sigma(b_{k-1}+1),\dots, \sigma(b_k).\]
 \end{definition}

 \begin{remark}
Since each $B_i$ consists of increasing entries, $\sigma$ does not affect the relative order of elements in the prime parking function $P_i$, and thus by Theorem \ref{thm: upf decomposition}, $\psi$ is well-defined into $\uIPFn$. 
 \end{remark}

\begin{example}\label{ex: bijection}

For example, consider the ordered partition $4/13/2$ consisting of blocks of size one, two, and one, respectively. 

By Definition~\ref{def:T_n to UPFn}:

$P_1=1$, $P_2=11$, $P_3=1$, and $\sigma=(4,1,3,2)$ so following Definition \ref{def:T_n to UPFn} we have  \[\psi(4/13/2)=\sigma(1|11|1)=\sigma(1,2,2,4)=(2,4,2,1).\] Note that the parking outcome of $(2,4,2,1)$ is $\sigma^{-1}=(2,4,3,1)$ and the breakpoints are at $1,3,4$, so following Definition \ref{def: UPFn to Tn} we have \[\phi(2,4,2,1)=\sigma(1)/\sigma(2)\,\sigma(3)/\sigma(4)=4/13/2.\]

\end{example}

 \begin{theorem}
 The maps $\psi$ and $\phi$ are inverses and establish a bijection between $\T_n$ and $\uIPFn$. 
 \end{theorem}
 \begin{proof}
 
Let $\alpha\in\uIPFn$ with parking outcome $\sigma^{-1}=(s_1,\dots, s_n) $ and break points $b_1,\dots,b_k$. By Theorem \ref{thm: upf decomposition} the parking order of $\alpha$, denoted $\alpha'$, has unique prime decomposition $P_1|\dots|P_k$ where $P_i=(1,1,\dots, b_i-b_{i-1}-1)$, where we let $b_0=0$. Considering $\psi\circ \phi(\alpha)$ we have 
\[\phi(\alpha)=\sigma(1),\dots,\sigma(b_1)/\sigma(b_1+1),\dots,\sigma(b_2)/\dots/\sigma(b_{k-1}+1),\dots, \sigma(b_k).\]
Then $\psi(\phi(\alpha))$ is $\sigma(P_1|\dots|P_k)$. By the uniqueness of the prime parking decomposition, $\sigma(P_1|\dots|P_k)=\alpha$.

Conversely, suppose $B_1/B_2/\dots/B_k\in \T_n$ consider $\phi\circ \psi(B_1/B_2/\dots/B_k)$. 
We have $\psi(B_1/B_2/\dots/B_k)=\sigma(P_1|P_2|\dots|P_k)$, where $P_i=(1,1,2,\dots,|B_i|-1)$ so that the parking function has break points $b_1,b_2,\dots, b_k$, where $b_i=|B_1|+|B_2|+\dots+|B_i|$. Hence, 
\begin{align*}
\phi(\psi(\alpha))&=\sigma(1),\dots,\sigma(b_1)/\sigma(b_1+1),\dots,\sigma(b_2)/\dots/\sigma(b_{k-1}+1),\dots, \sigma(b_k)\\
    &=B_1/B_2/\dots/B_k.
\end{align*}
Therefore, the functions $\psi$ and $\phi$ are inverses of each other, which establishes a bijection between $\T_n$ and $\uIPFn$. 
 \end{proof}

\begin{remark}[Fubini Rankings]\label{rem: fubini ranking explanation}
With slight modification to Definition \ref{def:T_n to UPFn}, we can construct an explicit bijection from faces of the permutohedron to Fubini rankings (weak orders), as well. Instead of associating to a block the distinguished prime unit-interval parking function from Lemma \ref{lem: unique_prime_unit_interval_pf}, we may use the all $1$'s vector of the appropriate length. Then, using analogs for breakpoints and parking outcomes for Fubini rankings, the inverse map would also behave similarly. This would establish the bijection remarked on in \cite{ovchinnikov2004weak} between $k$-weak orders and $(n-k)$-faces of the permutohedron. 
    
\end{remark}

 \begin{remark}
    
We may depict an ordered set partition $B_1/B_2/\dots/B_k$ as an $(n-k)\times n$ grid with squares filled in at row $i$ and column $j$ whenever $j\in B_i$. The ordered set partition and the corresponding labeled Dyck path for the associated unit-interval parking function from Example \ref{ex: bijection} are illustrated in Figure \ref{figure: grid}.
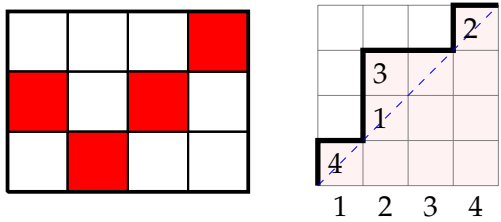
\begin{figure}[H] 
\centering
\begin{tikzpicture}[box/.style={rectangle,draw=black,thick, minimum size=1cm}, scale = 0.8, every node/.style={scale=0.8}, baseline = -0.05cm]

\foreach \x in {1, 2, 3, 4}{
    \foreach \y in {1, 2, 3}
        \node[box] at (\x,\y){};
}

\node[box,fill= red] at (1, 2){};  
\node[box,fill= red] at (2, 1){};  
\node[box,fill= red] at (3, 2){};
\node[box,fill= red] at (4, 3){};
\draw[line width = 0.5mm] (.5, .5) -- (4.5, .5) -- (4.5, 3.5) -- (.5, 3.5) -- (.5, .5);

\end{tikzpicture}
\qquad
\begin{tikzpicture}[scale = 0.6]
  \dyckpath{0, 0}{4}{1, 0, 1, 1, 0, 0, 1, 0}{dyck}{4};
  \draw[color = blue, dashed] (0, 0) -- (4, 4);
  \node[label=left:$4$] (b1) at (1, 0.5) {};
  \node[label=left:$1$] (b2) at (2, 1.5) {};
  \node[label=left:$3$] (b3) at (2, 2.5) {};
  \node[label=left:$2$] (b4) at (4, 3.5) {};
\end{tikzpicture}
\caption{Ordered set partition $4/13/2$ and the corresponding labeled Dyck path for parking function $(2, 4, 2, 1)$.}
\label{figure: grid}
\end{figure}

\end{remark}

\begin{theorem}\label{thm: pf to pphedron}
For all integers $0\leq k\leq n$, there exists a bijection between the $k$-dimensional faces of the permutohedron of order $n$ and unit-interval parking functions of length $n$ with total displacement $k$.
\end{theorem}

\begin{proof}

The map $\psi$ establishes the bijective correspondence between unit-interval parking functions with displacement $k$ to ordered set partitions with $n-k$ parts. This set is in natural bijection  with the $k$-faces of the permutohedron of order $n$, as shown in \cite{MartinAlgebraicCombinatorics,oldbook}.
\end{proof}

It is well-known that in the permutohedron of order $n$, denoted $P(n)$, the combinatorial type of each face is determined by the sizes of the blocks in the corresponding ordered set partition. Thus, Theorem~\ref{thm: pf to pphedron} immediately implies the following result. \begin{corollary}Each unit-interval parking function of length $n$ with component primes of size $n_1,\dots,n_k$, corresponds to a unique $(n-k)$-face of combinatorial type $P(n_1)\times\dots \times P(n_k)$ in $P(n)$, the permutohedron order $n$. 
\end{corollary}

We illustrate these results by considering the permutohedron of order 4.
\begin{example}[Permutohedron of order 4]
In Figure \ref{fig: 4-permutohedron}, the vertices of $P(4)$ are precisely the unit-interval parking functions of length 4 with displacement $0$, i.e. the permutations of $[4]$.
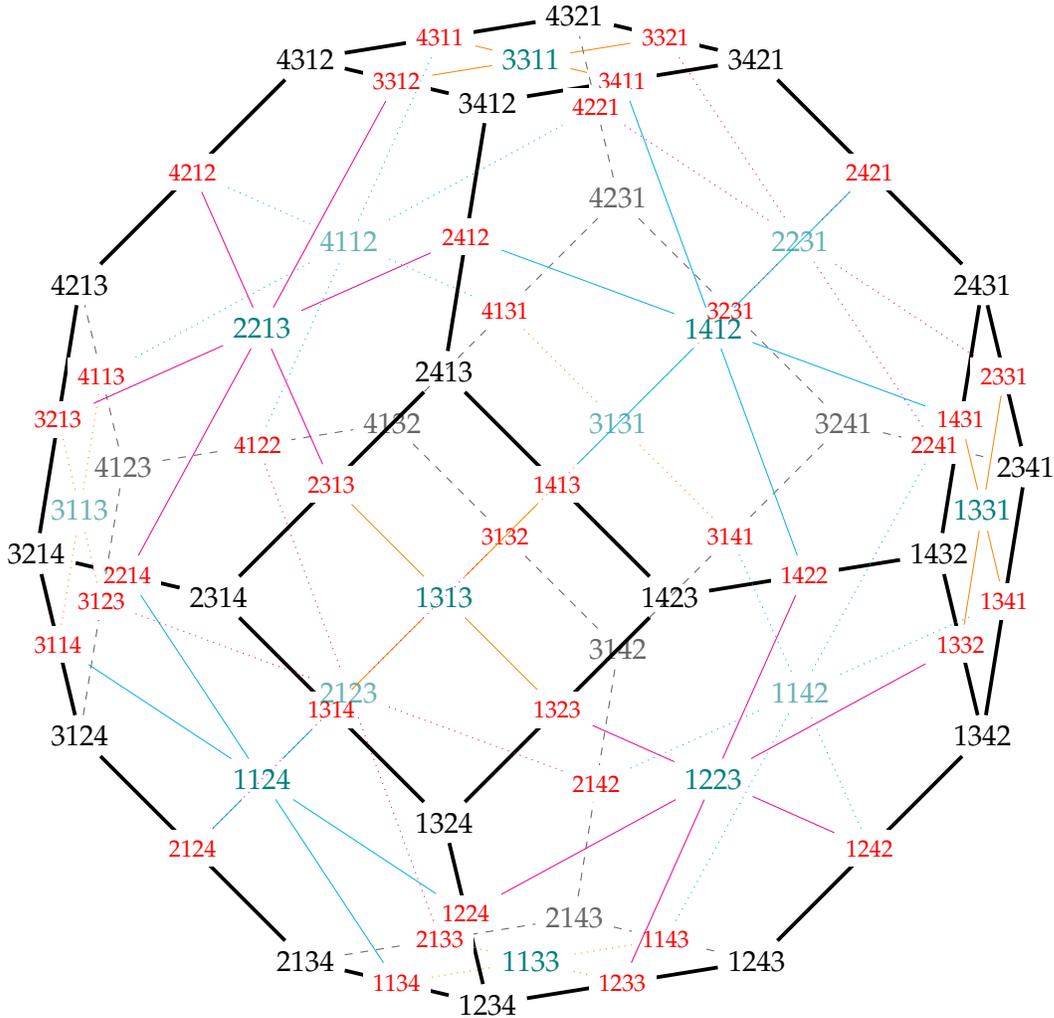
\begin{figure}
\centering
\begin{tikzpicture}[z={(0,0,.5)}, line join=round, scale = 3]


\node[fill = white] at (0, 1, 2) (2413) {$2413$};
\node at (0, -1, 2) (1324) {$1324$};
\node[opacity = 0.6] at (0, 1, -2) (4231) {$4231$};
\node[opacity = 0.6] at (0, -1, -2) (3142) {$3142$};

\node at (0, 2, 1) (3412) {$3412$};
\node at (0, -2, 1) (1234) {$1234$};
\node at (0, 2, -1) (4321) {$4321$};
\node[opacity = 0.6] at (0, -2, -1) (2143) {$2143$};

\node at (1, 0, 2) (1423) {$1423$};
\node at (-1, 0, 2) (2314) {$2314$};
\node[opacity = 0.6] at (1, 0, -2) (3241) {$3241$};
\node[opacity = 0.6] at (-1, 0, -2) (4132) {$4132$};

\node at (1, 2, 0) (3421) {$3421$};
\node at (-1, 2, 0) (4312) {$4312$};
\node at (1, -2, 0) (1243) {$1243$};
\node at (-1, -2, 0) (2134) {$2134$};

\node at (2, 0, 1) (1432) {$1432$};
\node at (-2, 0, 1) (3214) {$3214$};
\node at (2, 0, -1) (2341) {$2341$};
\node[opacity = 0.6] at (-2, 0, -1) (4123) {$4123$};

\node at (2, 1, 0) (2431) {$2431$};
\node at (-2, 1, 0) (4213) {$4213$};
\node at (2, -1, 0) (1342) {$1342$};
\node at (-2, -1, 0) (3124) {$3124$};

\draw[line width = 0.5mm] (1243) -- (1234) node [color = red, midway, fill=white] (1233) {\footnotesize $1233$} -- (2134) node [color = red, midway, fill=white] (1134) {\footnotesize $1134$};

\draw[line width = 0.5mm] (2134) -- (3124) node [color = red, midway, fill=white] (2124) {\footnotesize $2124$} -- (3214) node [color = red, midway, fill=white] (3114) {\footnotesize $3114$} -- (2314) node [color = red, midway, fill=white] (2214) {\footnotesize $2214$} -- (1324) node [color = red, midway, fill=white] (1314) {\footnotesize $1314$} -- (1234) node [color = red, midway, fill=white] (1224) {\footnotesize $1224$};

\draw[line width = 0.5mm] (1243) -- (1342) node [color = red, midway, fill=white] (1242) {\footnotesize $1242$} -- (1432) node [color = red, midway, fill=white] (1332) {\footnotesize $1332$} -- (1423) node [color = red, midway, fill=white] (1422) {\footnotesize $1422$} -- (1324) node [color = red, midway, fill=white] (1323) {\footnotesize $1323$};

\draw[line width = 0.5mm] (2314) -- (2413) node [color = red, midway, fill=white] (2313) {\footnotesize $2313$} -- (1423) node [color = red, midway, fill=white] (1413) {\footnotesize $1413$};

\draw[line width = 0.5mm] (4312) -- (4321) node [color = red, midway, fill=white] (4311) {\footnotesize $4311$} -- (3421) node [color = red, midway, fill=white] (3321) {\footnotesize $3321$} -- (3412) node [color = red, midway, fill=white] (3411) {\footnotesize $3411$} -- (4312) node [color = red, midway, fill=white] (3312) {\footnotesize $3312$};

\draw[line width = 0.5mm] (3214) -- (4213) node [color = red, midway, fill=white] (3213) {\footnotesize $3213$} -- (4312) node [color = red, midway, fill=white] (4212) {\footnotesize $4212$};

\draw[line width = 0.5mm] (3412) -- (2413) node [color = red, midway, fill=white] (2412) {\footnotesize $2412$};

\draw[line width = 0.5mm] (3421) -- (2431) node [color = red, midway, fill=white] (2421) {\footnotesize $2421$}-- (1432) node [color = red, midway, fill=white] (1431) {\footnotesize $1431$};

\draw[line width = 0.5mm] (2431) -- (2341) node [color = red, midway, fill=white] (2331) {\footnotesize $2331$} -- (1342) node [color = red, midway, fill=white] (1341) {\footnotesize $1341$};

\draw[dashed, opacity = 0.6] (2134) -- (2143) node [color = red, fill = white, text opacity = 0.6, fill opacity = 1, midway] (2133) {\footnotesize $2133$} -- (1243) node [text opacity = 0.6, fill opacity = 1,, color = red, midway, fill=white] (1143) {\footnotesize $1143$};

\draw[dashed, opacity = 0.6] (3124) -- (4123) node [text opacity = 0.6, fill opacity = 1, color = red, midway, fill=white] (3123) {\footnotesize $3123$} -- (4213) node [text opacity = 0.6, fill opacity = 1, color = red, midway, fill=white] (4113) {\footnotesize $4113$};

\draw[dashed, opacity = 0.6] (4123) -- (4132) node [text opacity = 0.6, fill opacity = 1, color = red, midway, fill=white] (4122) {\footnotesize $4122$} -- (3142) node [text opacity = 0.6, fill opacity = 1, color = red, midway, fill=white] (3132) {\footnotesize $3132$} -- (2143) node [text opacity = 0.6, fill opacity = 1, color = red, midway, fill=white] (2142) {\footnotesize $2142$};

\draw[dashed, opacity = 0.6] (4132) -- (4231) node [text opacity = 0.6, fill opacity = 1, color = red, midway, fill=white] (4131) {\footnotesize $4131$} -- (3241) node [text opacity = 0.6, fill opacity = 1, color = red, midway, fill=white] (3231) {\footnotesize $3231$} -- (3142) node [text opacity = 0.6, fill opacity = 1, color = red, midway, fill=white] (3141) {\footnotesize $3141$};

\draw[dashed, opacity = 0.6] (4231) -- (4321) node [text opacity = 0.6, fill opacity = 1, color = red, midway, fill=white] (4221) {\footnotesize $4221$};

\draw[dashed, opacity = 0.6] (3241) -- (2341) node [text opacity = 0.6, fill opacity = 1, color = red, midway, fill=white] (2241){\footnotesize $2241$};

\node[color = teal] at (0, 0, 2) (1313) {$1313$};
\node[color = teal, opacity = 0.6] at (0, 0, -2) (3131) {$3131$};
\node[color = teal] at (0, 2, 0) (3311) {$3311$};
\node[color = teal] at (0, -2, 0) (1133) {$1133$};
\node[color = teal] at (2, 0, 0) (1331) {$1331$};
\node[color = teal, opacity = 0.6] at (-2, 0, 0) (3113) {$3113$};

\node[color = teal] at (1, 1, 1) (1412) {$1412$};
\node[color = teal, opacity = 0.6] at (1, 1, -1) (2231) {$2231$};
\node[color = teal] at (1, -1, 1) (1223) {$1223$};
\node[color = teal, opacity = 0.6] at (1, -1, -1) (1142) {$1142$};
\node[color = teal] at (-1, 1, 1) (2213) {$2213$};
\node[color = teal, opacity = 0.6] at (-1, 1, -1) (4112) {$4112$};
\node[color = teal] at (-1, -1, 1) (1124) {$1124$};
\node[color = teal, opacity = 0.6] at (-1, -1, -1) (2123) {$2123$};

\draw[color = orange, line width = 0.05mm] (2313) -- (1313);
\draw[color = orange, line width = 0.05mm] (1413) -- (1313);
\draw[color = orange, line width = 0.05mm] (1323) -- (1313);
\draw[color = orange, line width = 0.05mm] (1314) -- (1313);

\draw[color = orange, line width = 0.05mm] (3312) -- (3311);
\draw[color = orange, line width = 0.05mm] (4311) -- (3311);
\draw[color = orange, line width = 0.05mm] (3321) -- (3311);
\draw[color = orange, line width = 0.05mm] (3411) -- (3311);

\draw[color = orange, line width = 0.05mm] (1431) -- (1331);
\draw[color = orange, line width = 0.05mm] (2331) -- (1331);
\draw[color = orange, line width = 0.05mm] (1341) -- (1331);
\draw[color = orange, line width = 0.05mm] (1332) -- (1331);

\draw[color = orange, dotted] (4131) -- (3131);
\draw[color = orange, dotted] (3231) -- (3131);
\draw[color = orange, dotted] (3141) -- (3131);
\draw[color = orange, dotted] (3132) -- (3131);

\draw[color = orange, dotted] (3213) -- (3113);
\draw[color = orange, dotted] (3114) -- (3113);
\draw[color = orange, dotted] (3123) -- (3113);
\draw[color = orange, dotted] (4113) -- (3113);

\draw[color = orange, dotted] (1134) -- (1133);
\draw[color = orange, dotted] (2133) -- (1133);
\draw[color = orange, dotted] (1143) -- (1133);
\draw[color = orange, dotted] (1233) -- (1133);


\draw[color = cyan, line width = 0.05mm] (2124) -- (1124);
\draw[color = cyan, line width = 0.05mm] (3114) -- (1124);
\draw[color = cyan, line width = 0.05mm] (2214) -- (1124);
\draw[color = cyan, line width = 0.05mm] (1314) -- (1124);
\draw[color = cyan, line width = 0.05mm] (1224) -- (1124);
\draw[color = cyan, line width = 0.05mm] (1134) -- (1124);

\draw[color = cyan, line width = 0.05mm] (1422) -- (1412);
\draw[color = cyan, line width = 0.05mm] (1413) -- (1412);
\draw[color = cyan, line width = 0.05mm] (2412) -- (1412);
\draw[color = cyan, line width = 0.05mm] (3411) -- (1412);
\draw[color = cyan, line width = 0.05mm] (2421) -- (1412);
\draw[color = cyan, line width = 0.05mm] (1431) -- (1412);

\draw[color = magenta, line width = 0.05mm] (1233) -- (1223);
\draw[color = magenta, line width = 0.05mm] (1224) -- (1223);
\draw[color = magenta, line width = 0.05mm] (1323) -- (1223);
\draw[color = magenta, line width = 0.05mm] (1422) -- (1223);
\draw[color = magenta, line width = 0.05mm] (1332) -- (1223);
\draw[color = magenta, line width = 0.05mm] (1242) -- (1223);

\draw[color = magenta, line width = 0.05mm] (2214) -- (2213);
\draw[color = magenta, line width = 0.05mm] (3213) -- (2213);
\draw[color = magenta, line width = 0.05mm] (4212) -- (2213);
\draw[color = magenta, line width = 0.05mm] (3312) -- (2213);
\draw[color = magenta, line width = 0.05mm] (2412) -- (2213);
\draw[color = magenta, line width = 0.05mm] (2313) -- (2213);

\draw[color = cyan, dotted] (1143) -- (1142);
\draw[color = cyan, dotted] (2142) -- (1142);
\draw[color = cyan, dotted] (3141) -- (1142);
\draw[color = cyan, dotted] (2241) -- (1142);
\draw[color = cyan, dotted] (1341) -- (1142);
\draw[color = cyan, dotted] (1242) -- (1142);

\draw[color = cyan, dotted] (4122) -- (4112);
\draw[color = cyan, dotted] (4113) -- (4112);
\draw[color = cyan, dotted] (4212) -- (4112);
\draw[color = cyan, dotted] (4311) -- (4112);
\draw[color = cyan, dotted] (4221) -- (4112);
\draw[color = cyan, dotted] (4131) -- (4112);

\draw[color = magenta, dotted] (2133) -- (2123);
\draw[color = magenta, dotted] (2124) -- (2123);
\draw[color = magenta, dotted] (3123) -- (2123);
\draw[color = magenta, dotted] (4122) -- (2123);
\draw[color = magenta, dotted] (3132) -- (2123);
\draw[color = magenta, dotted] (2142) -- (2123);

\draw[color = magenta, dotted] (3231) -- (2231);
\draw[color = magenta, dotted] (4221) -- (2231);
\draw[color = magenta, dotted] (3321) -- (2231);
\draw[color = magenta, dotted] (2421) -- (2231);
\draw[color = magenta, dotted] (2331) -- (2231);
\draw[color = magenta, dotted] (2241) -- (2231);

\end{tikzpicture}
\caption{A permutohedron of order $4$ with parking functions labeling all $k$-faces.}
\label{fig: 4-permutohedron}
\end{figure}

An edge of $P(4)$ has combinatorial type corresponding to a product $P(1)\times P(1)\times P(2)$. Accordingly every edge can be uniquely labeled with one of the parking functions of length $4$ with displacement $1$.

However, something interesting happens at the $2$-faces: we see two different shapes. On one hand the face $2213$ corresponds to a face with combinatorial type $P(1)\times P(3)$, which is hexagonal face. On the other, $1313$ corresponds to a face with combinatorial type $P(2)\times P(2)$. The only unit-interval parking functions with displacement partition $(1,1)$ consist of either a single nontrivial prime parking function of length three or two nontrivial prime parking functions, each of length two.

\end{example}

Recall, the symmetric group acts simply transitively on the vertices of the permutohedron, but does not act transitively nor freely on the set of $k$-faces. 
Using the bijection from Theorem \ref{thm: pf to pphedron}, we can cleanly illuminate the action on the faces of the permutohedron through the lens of unit-interval parking functions. 
In particular, Theorem~\ref{thm: pf to pphedron} induces an action for the symmetric group on $\uIPFn$. 
Moreover, the known properties for the action on the permutohedron translated into unit-interval parking functions, carry along numerical results.

Observe that under this action, a permutation $\sigma$ fixes a face when the set of vertices of that face is invariant under $
\sigma$. Since each face can be thought of as a product of lower order permutohedra, this means the stabilizer of any face is isomorphic to the product of the stabilizers of the factors.  Since the prime decomposition of a unit-interval parking function determines the combinatorial type of a face,
the stabilizer of any face is determined by partitioning the label's preferences into its prime components. The permutations that only permute indices within primes will fix the face. Before we formalize this notion in Proposition \ref{prop:Stabilizer} below, we present the following technical result. 
\begin{lemma}
Two faces of $P(n)$ are in the same orbit if and only if their corresponding parking function labels have the same ordered prime decomposition.
\end{lemma}
\begin{proof}First observe that two faces of $P(n)$ are in the same orbit if and only if their corresponding ordered set partitions have blocks of the same size in the same order. Applying the bijection from  Theorem~\ref{thm: pf to pphedron}, the parking functions have the same preferences. Two parking functions have the same preferences if and only if they have the same ascending order. By Remark \ref{rem: pp_order_unit_interval}, unit-interval parking functions have the same ascending order if and only if they have the same parking order, which determines the ordered prime decomposition. This gives our characterization of the orbits of faces of the permutohedron under the symmetric group action.
\end{proof}

The following proposition follows directly from the bijection between the $k$-dimensional faces of the permutohedron of order $n$ and unit-interval parking functions of length $n$ with total displacement $k$.
\begin{proposition}\label{prop:Stabilizer}
Let $\alpha$ be a unit-interval parking function that decomposes into $k$ primes, let $e(\alpha_{i,j})$ be the entry of the $i$th preference of the $j$th prime in $\alpha$'s prime decomposition, and let $\ell(j)$ be the length of the $j$th prime. 
Then the transpositions of the form $(e(\alpha_{i,j}),e(\alpha_{i+1,j}))$ for $j \in [k]$  and $i \in [\ell(j)-1]$ generate the stabilizer of the face that $\alpha$ labels, and the size of the stabilizer of that face is $\prod_{j\in [k]} \ell(j)!$.
\end{proposition}
From the orbit-stabilizer theorem we can directly compute the size each orbit in terms of the prime parking decomposition of a unit-interval parking function.

\begin{corollary}\label{cor: orbit}
Let $\alpha$ be defined in Proposition \ref{prop:Stabilizer}.
   The size of the orbit of each unit-interval parking function is \[ \frac{n!}{\prod_{j\in [k]} \ell(j)!}.\]
    
\end{corollary}

\begin{figure}
\begin{tikzpicture}[scale = 4]
\node at (0.5, 0.866, 0.5) (23145) {$23145$};
\node at (1, 0.866, -0.5) (23154) {$23154$};
\node at (0.5, -0.866, 0.5) (12345) {$12345$};
\node at (1, -0.866, -0.5) (12354) {$12354$};
\node at (-0.5, 0.866, 0.5) (32145) {$32145$};
\node at (0, 0.866, -0.5) (32154) {$32154$};
\node at (-0.5, -0.866, 0.5) (21345) {$21345$};
\node[opacity = 0.6] at (0, -0.866, -0.5) (21354) {$21354$};

\node at (1, 0, 0.5) (13245) {$13245$};
\node at (1.5, 0, -0.5) (13254) {$13254$};
\node at (-1, 0, 0.5) (31245) {$31245$};
\node[opacity = 0.6] at (-.5, 0, -0.5) (31254) {$31254$};


\draw[line width = 0.5mm] (23145) -- (32145) node [color = red, midway, fill=white] (22145) {\footnotesize $22145$} -- (31245) node [color = red, midway, fill=white] (31145) {\footnotesize $31145$} -- (21345) node [color = red, midway, fill=white] (21245) {\footnotesize $21245$} -- (12345) node [color = red, midway, fill=white] (11345) {\footnotesize $11345$} -- (13245) node [color = red, near end, fill=white] (12245) {\footnotesize $12245$} -- (23145) node [color = red, near end, fill=white] (13145) {\footnotesize $13145$};

\draw[line width = 0.5mm] (23154) -- (32154) node [color = red, midway, fill=white] (22154) {\footnotesize $22154$};

\draw[dashed, opacity = 0.6] (32154) -- (31254) node [color = red, near end, fill=white, fill opacity = 1, text opacity = 0.6] (31154) {\footnotesize $31154$} -- (21354) node [color = red, near end, fill=white, fill opacity = 1, text opacity = 0.6] (21254) {\footnotesize $21254$} -- (12354) node [color = red, midway, fill=white, fill opacity = 1, text opacity = 0.6] (11354) {\footnotesize $11354$};

\draw[line width = 0.5mm] (12354) -- (13254) node [color = red, midway, fill=white] (12254) {\footnotesize $12254$} -- (23154) node [color = red, midway, fill=white] (13154) {\footnotesize $13154$};

\draw[line width = 0.5mm] (23145) -- (23154) node [color = red, midway, fill=white] (23144) {\footnotesize $23144$};

\draw[line width = 0.5mm] (32145) -- (32154) node [color = red, midway, fill=white] (32144) {\footnotesize $32144$};

\draw[dashed, opacity = 0.6] (31245) -- (31254) node [color = red, midway, fill=white, fill opacity = 1, text opacity = 0.6] (31244) {\footnotesize $31244$};

\draw[dashed, opacity = 0.6] (21345) -- (21354) node [color = red, midway, fill=white, fill opacity = 1, text opacity = 0.6] (21344) {\footnotesize $21344$};

\draw[line width = 0.5mm] (12345) -- (12354) node [color = red, midway, fill=white] (12344) {\footnotesize $12344$};

\draw[line width = 0.5mm] (13245) -- (13254) node [color = red, midway, fill=white] (13244) {\footnotesize $13244$};


\node[color = teal, fill = white] at (0, 0, 0.5) (11245) {$11245$};
\node[color = teal, opacity = 0.6] at (0.5, 0, -0.5) (11254) {$11254$};
\node[color = teal, fill = white] at (1, 0.433, 0) (13144) {$13144$};
\node[color = teal, fill = white] at (1, -0.433, 0) (12244) {$12244$};
\node[color = teal, fill = white] at (-0.5, 0.433, 0) (31144) {$31144$};
\node[color = teal, fill = white] at (-0.5, -0.433, 0) (21244) {$21244$};
\node[color = teal, fill = white] at (.25, -.866, 0) (11344) {$11344$};
\node[color = teal, fill = white] at (.25, .866, 0) (22144) {$22144$};

\draw[color = cyan] (21245) -- (11245);
\draw[color = cyan] (31145) -- (11245);
\draw[color = cyan] (22145) -- (11245);
\draw[color = cyan] (13145) -- (11245);
\draw[color = cyan] (12245) -- (11245);
\draw[color = cyan] (11345) -- (11245);

\draw[color = cyan, dotted] (13154) -- (11254);
\draw[color = cyan, dotted] (22154) -- (11254);
\draw[color = cyan, dotted] (31154) -- (11254);
\draw[color = cyan, dotted] (21254) -- (11254);
\draw[color = cyan, dotted] (11354) -- (11254);
\draw[color = cyan, dotted] (12254) -- (11254);

\draw[color = orange] (12344) -- (12244);
\draw[color = orange] (12245) -- (12244);
\draw[color = orange] (13244) -- (12244);
\draw[color = orange] (12254) -- (12244);

\draw[color = orange] (13145) -- (13144);
\draw[color = orange] (23144) -- (13144);
\draw[color = orange] (13154) -- (13144);
\draw[color = orange] (13244) -- (13144);

\draw[color = orange] (22145) -- (22144);
\draw[color = orange] (32144) -- (22144);
\draw[color = orange] (22154) -- (22144);
\draw[color = orange] (23144) -- (22144);

\draw[color = orange, dotted] (32144) -- (31144);
\draw[color = orange, dotted] (31145) -- (31144);
\draw[color = orange, dotted] (31244) -- (31144);
\draw[color = orange, dotted] (31154) -- (31144);

\draw[color = orange, dotted] (31244) -- (21244);
\draw[color = orange, dotted] (21245) -- (21244);
\draw[color = orange, dotted] (21344) -- (21244);
\draw[color = orange, dotted] (21254) -- (21244);

\draw[color = orange, dotted] (21344) -- (11344);
\draw[color = orange, dotted] (11345) -- (11344);
\draw[color = orange, dotted] (12344) -- (11344);
\draw[color = orange, dotted] (11354) -- (11344);

\node[color = purple, opacity = 0.6] at (0.25, 0, 0) (11244) {$11244$};
\draw[color = purple, dotted] (11245) -- (11244);
\draw[color = purple, dotted] (11254) -- (11244);
\draw[color = purple, dotted] (11344) -- (11244);
\draw[color = purple, dotted] (12244) -- (11244);
\draw[color = purple, dotted] (13144) -- (11244);
\draw[color = purple, dotted] (22144) -- (11244);
\draw[color = purple, dotted] (31144) -- (11244);
\draw[color = purple, dotted] (21244) -- (11244);

\end{tikzpicture}
\caption{A hexagonal prism, the facet of $P(5)$ corresponding the the parking function $11244$.}
\label{fig: hexprism}
\end{figure}
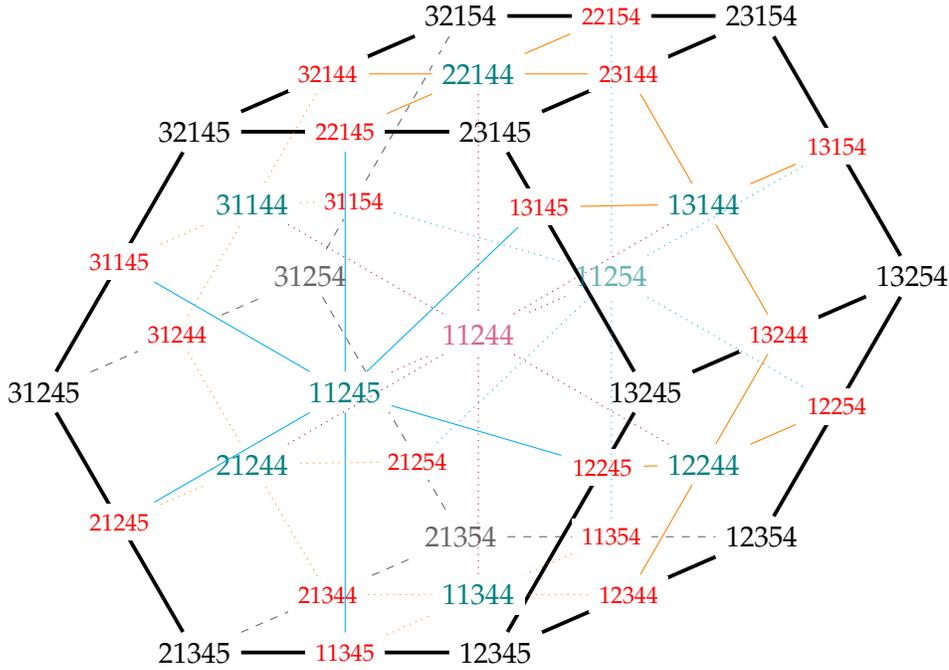
The formula in Corollary \ref{cor: orbit} counts the number of unit-interval parking functions that differ by a permutation. In a similar vein, we can enumerate the number of permutations that when applied to a fixed unit-inteval parking function result in another (not  necessarily distinct) unit-interval parking function by recognizing that a simple transposition of the first two entries in any nontrivial prime component fixes the parking function. In summary, we have the following result. 

\begin{corollary}\label{cor: perms that fix}
Let $\alpha$ be a unit-interval parking function that decomposes into $k$ primes including $t$ non-trivial primes, and let $\ell(j)$ be the length of the $j$th prime. The number of permutations that fix the displacement partition of $\alpha$ is \[\frac{2^tn!}{\prod_{j\in [k]} \ell(j)!}.\]
\end{corollary}
\begin{example}
    For example, the face in $P(5)$ defined by the unit-interval parking function $11244$ has an ordered prime decomposition into a length $3$ prime followed by a length $2$ prime. The symmetric group action can send this facet to any other of the ten unit-interval parking functions with the same ordered decomposition:
    \[44112, \quad 41412,  \quad 41142, \quad  41124, \quad 14412,  \quad 14142,  \quad 14124,  \quad 11442, \quad  11424.  \]
     Another facet with the same prime decomposition (and therefore shape) is labeled by $11334$. However, the order of the primes differ--- the length $2$ prime is ordered before the length $3$ prime, and no permutation will send $11244$ to $11334$, so they are not in the same orbit. Note that the number of permutations that preserve the orbit of the face labeled $11244$ is 40. The same is true for $11334$.
\end{example}

We conclude by remarking that unit-interval parking functions allow us to give a novel, simple, and elegant description of the symmetric group action on the permutohedron. Moreover, the origin of parking functions in linear probing establishes a relationship between the permutohedron and random hashing functions. Lastly, despite their unassuming nature, we have shown that unit-interval parking functions are a rich subset of parking functions with connections to a myriad of combinatorial objects.

 \begin{acknowledgements}
 This material is based upon work supported by the National Science Foundation under Grant No. DMS-1929284 while the authors were in residence at the Institute for Computational and Experimental Research in Mathematics in Providence, RI.
 P.~E.~Harris was partially supported by a Karen Uhlenbeck EDGE Fellowship.
\end{acknowledgements}
\bibliographystyle{amsplain}
\bibliography{bib}
\end{document}